\documentclass[12pt, letterpaper]{article}
\usepackage[T1]{fontenc}
\usepackage{mathptmx} 
\usepackage{amsmath,amssymb,amsthm,subfigure}
\usepackage[dvips]{epsfig}
\usepackage{titlesec}
\usepackage{setspace}
\usepackage{amsmath,amssymb,subfigure}
\usepackage[dvips]{epsfig}

\usepackage[margin=1.2in]{geometry}

\makeatletter \@addtoreset{equation}{section} \makeatother

\usepackage{amsmath} 
\usepackage{amsthm}
\usepackage{amssymb}
\usepackage{color}
\usepackage{pgf,tikz}
\usepackage{mathrsfs}
\usetikzlibrary{arrows}

\usepackage{epstopdf}
\usepackage{graphicx}
\graphicspath{{images/}}

\newtheorem{theorem}{Theorem}

\theoremstyle{definition}

\theoremstyle{remark}

\begin{document}

\definecolor{xfqqff}{rgb}{0.4980392156862745,0,1}\definecolor{qqqqff}{rgb}{0,0,1}\definecolor{wwccff}{rgb}{0.4,0.8,1}\definecolor{uuuuuu}{rgb}{0.26666666666666666,0.26666666666666666,0.26666666666666666}

\begin{center}
{\large Linear Upper Bound on the Ribbonlength of  Torus Knots and Twist Knots} 

\vspace*{.88in}

{Grace M. Tian \\ 

\vspace*{.18in}

Harvard University\\
gtian@college.harvard.edu}

\bigskip

\end{center}

\bigskip

\begin{abstract}

Knotted ribbons form an important topic in knot theory. They have applications in natural sciences, such as cyclic duplex DNA modeling. A flat knotted ribbon can be obtained by gently pulling a knotted ribbon tight so that it becomes flat and folded. An important problem in knot theory is to study the minimal ratio of  length to width of a flat knotted ribbon. This minimal ratio is called the ribbonlength of the knot. It has been conjectured that the ribbonlength has an upper bound and a lower bound which are both linear in the crossing number of the knot.

In the first part of the paper, we use grid diagrams to construct flat knotted ribbons and prove an explicit quadratic upper bound on the ribbonlength for all non-trivial knots. We then improve the quadratic upper bound to a linear upper bound for all non-trivial torus knots and twist knots. Our approach of using grid diagrams to study flat knotted ribbons is novel and can likely be used to obtain a linear upper bound for more general families of knots. In the second part of the paper, we obtain a sharper linear upper bound on the ribbonlength for nontrivial twist knots by constructing a flat knotted ribbon via folding the ribbon over itself multiple times to shorten the length.

\end{abstract}

\newpage

\setcounter{page}{2}
\setcounter{section}{0}

\nocite{*}
\section{Introduction}

Knots have been around for thousands of years, but they have only attracted the attention of mathematicians for a relatively short time. In 1867, the physicist William Thomson (later Lord Kelvin) \cite{T} proposed that atoms were knots in a medium called the ether. This led many scientists to believe that they could understand atoms by simply studying knots. Mathematicians began to classify knots and create tables of knots. Although the Michelson-Morley experiment of 1887 \cite{MM} dismissed the existence of the ether and thus the atom/knot hypothesis, knot theory has over time become a promising field of mathematical research in its own right.

A knot is a closed curve in $\mathbb{R}^3$ that is homeomorphic to a circle. Two knots are equivalent if one can continuously deform one knot into the other knot. An effective way to study knots is to consider the projection of a knot onto a plane. The knot diagram of a knot is a projection of the knot with additional information about overcrossing or undercrossing at each crossing point in the diagram. In 1927, Reidemeister \cite{R} showed that two knot diagrams represent the same knot if and only if they are related by a finite sequence of Reidemeister moves, which consist of a twist move, a poke move, and a slide move.

Knotted ribbons form an important topic in knot theory. A knotted ribbon is a topological object in $\mathbb{R}^3$ that is homeomorphic to an annulus.
It can be viewed as a closed strip of paper with a knot. Knotted ribbons find applications in natural sciences. They have been used to model the cyclic duplex DNA in molecular biology, with the two boundaries of the ribbon corresponding to the two edges of the DNA ladder \cite{A}. 

Flat knotted ribbons were first studied by Kauffman \cite{K} in 2005. Like a knot diagram, a flat knotted ribbon is also a two-dimensional object. Intuitively, a flat knotted ribbon can be obtained by gently pulling a knotted ribbon tight so that it becomes flat and folded. The {\em ribbonlength} of a knot is defined to be the minimal ratio of length to width of all possible flat knotted ribbons for the knot. An important problem in knot theory is to study the ribbonlength of a given knot. This question about the ribbonlength of a knot is analogous to a question about the ropelength of a knot, which is the minimal ratio of length to radial width for knotted tubes \cite{BS, CKS, K}. The latter question has immediate applications to biology and chemistry \cite{BS, CKS}. 

Unlike the ropelength of a knot, little has been done on the ribbonlength. Kauffman \cite{K} was able to calculate the ratio of length to width of certain flat knotted ribbons for trefoil knot and figure eight knot and hence obtained an upper bound on the ribbonlength of these two knots. Later, the ribbonlength was also studied by Kennedy, Mattman, Raya, and Tating \cite{KMRT} for torus knots. They also obtained an upper bound on the ribbonlength of some families of torus knots. Their upper bound is linear in the crossing number of the knot for $(q, q+1), (q, 2q + 1)$, and $(q , 2q +2)$ torus knots and is quadratic for $(2, q)$ torus knots.  The {\em crossing number} of a knot is the smallest number of crossings of all possible knot diagrams of the knot. More recently, E. Denne, M. Kamp, R. Terry, and X. Zhu \cite{DKTZ} studied how the number of line segments in a knot diagram formed  by a piecewise linear curve influences the ratio of length to width in the flat knotted ribbon.

Kusner \cite{K} conjectured that the ribbonlength of a knot has upper and lower bounds which are linear in the crossing number of the knot. The linear upper bounds  found by Kennedy, Mattman, Raya, and Tating \cite{KMRT} for some torus knots support Kusner's conjecture.  

In this paper, we use grid diagrams to study flat knotted ribbons. Unlike previous researchers who construct flat knotted ribbons via sophisticated folding, we form a flat knotted ribbon by thickening up the horizontal and vertical paths in the grid diagram of the knot (see Figure 1). For every non-trivial knot, we prove that the ribbonlength has an upper bound which is linear in the crossing number of the knot. We improve the quadratic upper bound to a linear upper bound for every non-trivial torus knot. Our results extend the linear upper bound results of  Kennedy, Mattman, Raya, and Tating \cite{KMRT} from $(q, q+1), (q, 2q + 1)$, and $(q , 2q +2)$ torus knots to all nontrivial torus knots. More importantly, we  improve their results on $(2, q)$ torus knots from a quadratic upper bound to a linear upper bound. For every nontrivial twist knot, we also use grid diagrams to show that the ribbonlength is bounded by an effective constant times the crossing number. We then obtain a sharper linear upper bound for twist knots by constructing a flat knotted ribbon via folding the ribbon over itself multiple times to minimize the length of the ribbon. 

Our approach of using grid diagrams to study flat knotted ribbons is novel and can obviously be used to obtain a linear upper bound on the ribbonlength for more general families of knots.

The rest of the paper is organized as follows. In Section 2, we use grid diagrams to study flat knotted ribbons. We construct a flat knotted ribbon by thickening up the horizontal and vertical paths in the grid diagram, calculate the ratio of length to width for the flat knotted ribbon and hence obtain an upper bound on the ribbonlength. We exploit a known relation between the grid index and crosssing number of a knot and prove a quadratic upper bound on ribbonlength for all non-trivial knots. We then improve the quadratic upper bound to a linear upper bound for all non-trivial torus and twist knots. In Section 3, we obtain a shaper linear upper bound on the ribbonlength of twist knots by constructing a flat knotted ribbon via folding the ribbon over itself multiple times. Finally, we conclude the paper with directions for future research in Section 4. 

\section{Linear upper bound for torus knots and twist knots}
A flat knotted ribbon can be obtained by gently pulling a knotted ribbon tight so that it becomes flat and folded. The center-line of the flat knotted ribbon then becomes a piecewise linear curve in the plane, which forms a knot diagram for the corresponding knot. Each vertex of the piecewise linear curve corresponds to a fold in the flat knotted ribbon.
Conversely, given a knot diagram in the form of a piecewise linear curve, one can construct a corresponding flat knotted ribbon by thickening up the line segments in the knot diagram.

We now consider knot diagrams formed by horizontal and vertical line segments. A grid diagram \cite{OSS} consists of a square grid on the plane with  $n \times n$ unit square cells and a collection of black and white dots. These dots are arranged so that (see Figure 1)

\begin{itemize}

\item every row and column contains exactly one black dot and one white dot;

\item  no cell contains more than one dot;

\item the dot is in the middle of the cell.

\end{itemize}

\bigskip

\begin{figure}[!htbp]
\centering
\includegraphics[width=0.28\textwidth]{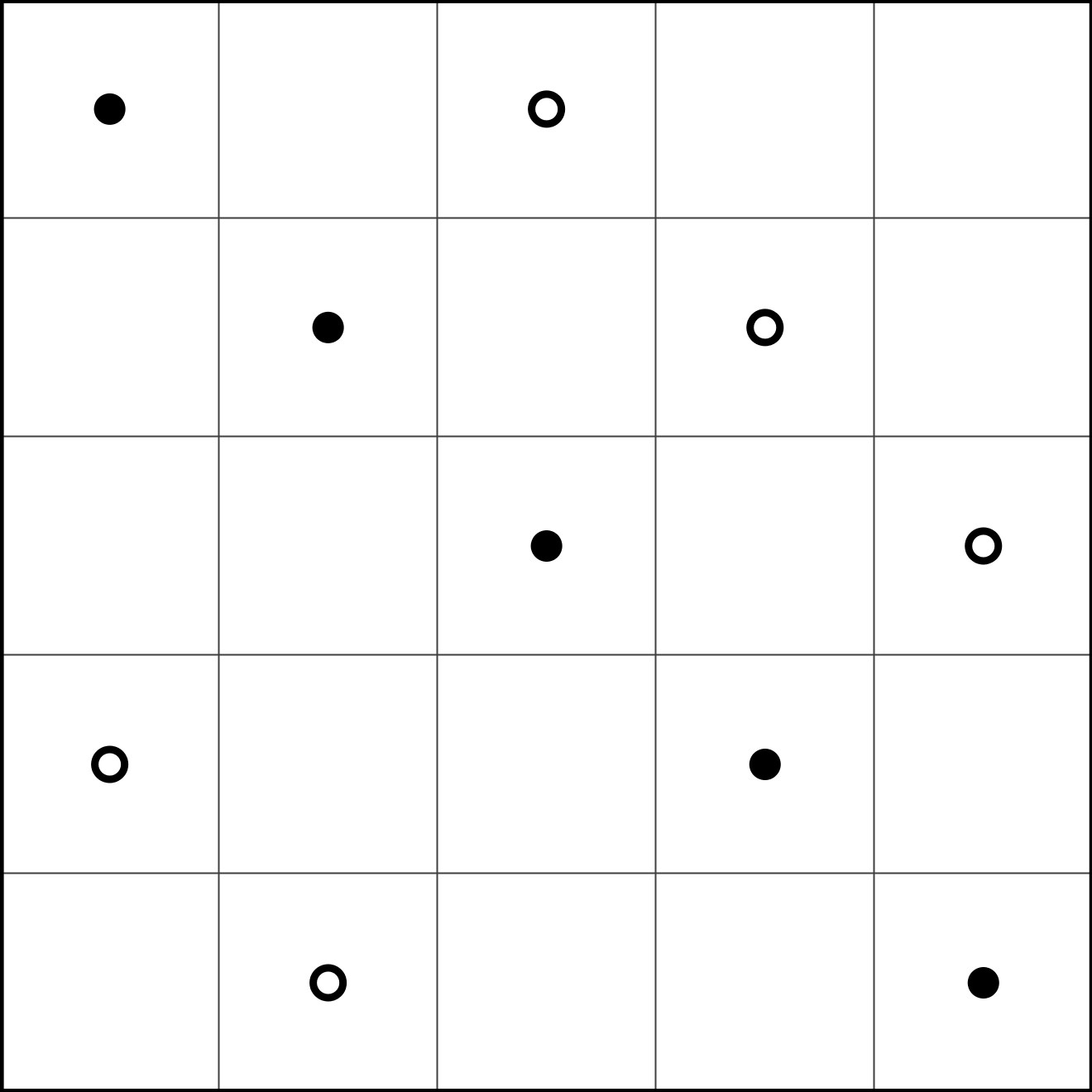}
\includegraphics[width=0.28\textwidth]{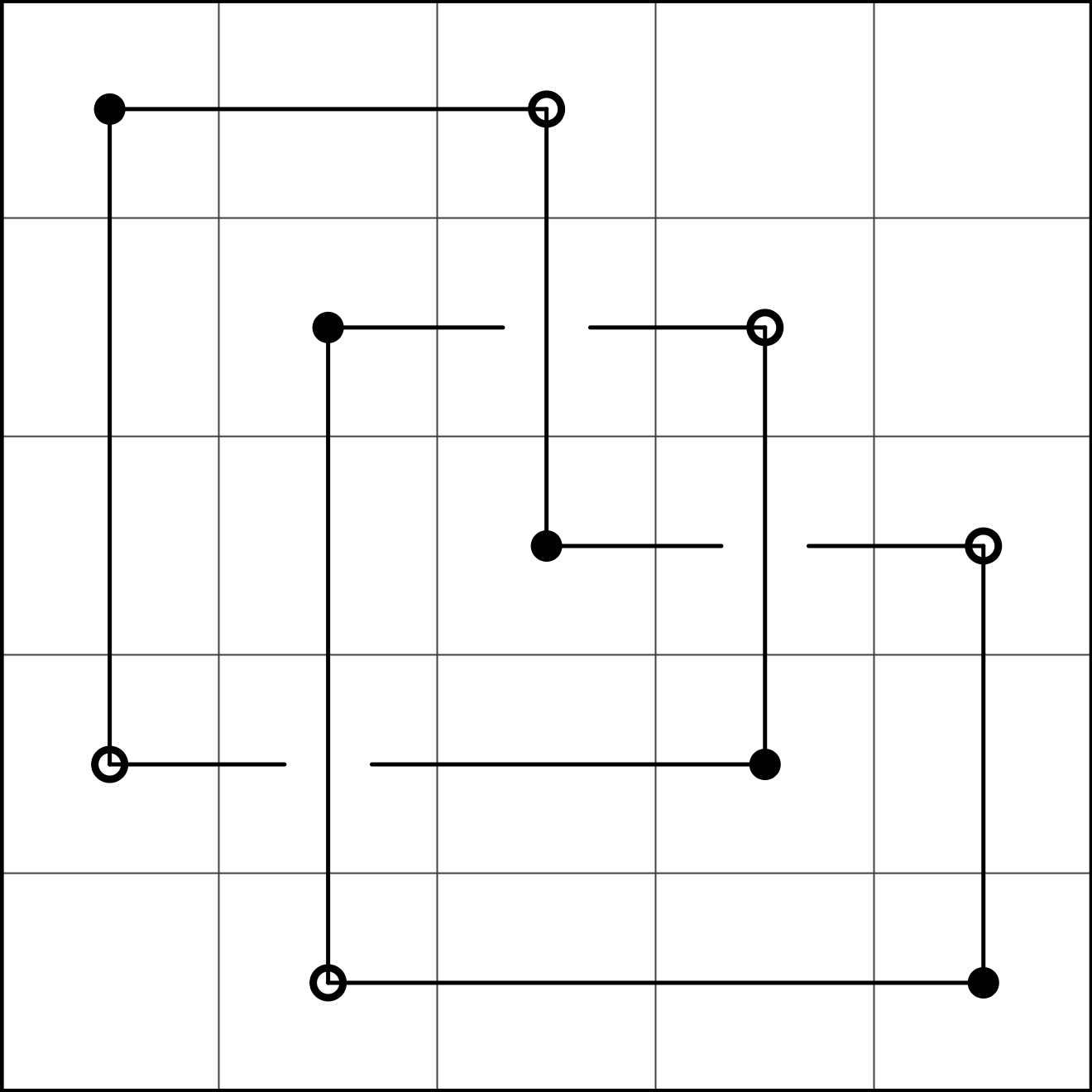}
\includegraphics[width=0.28\textwidth]{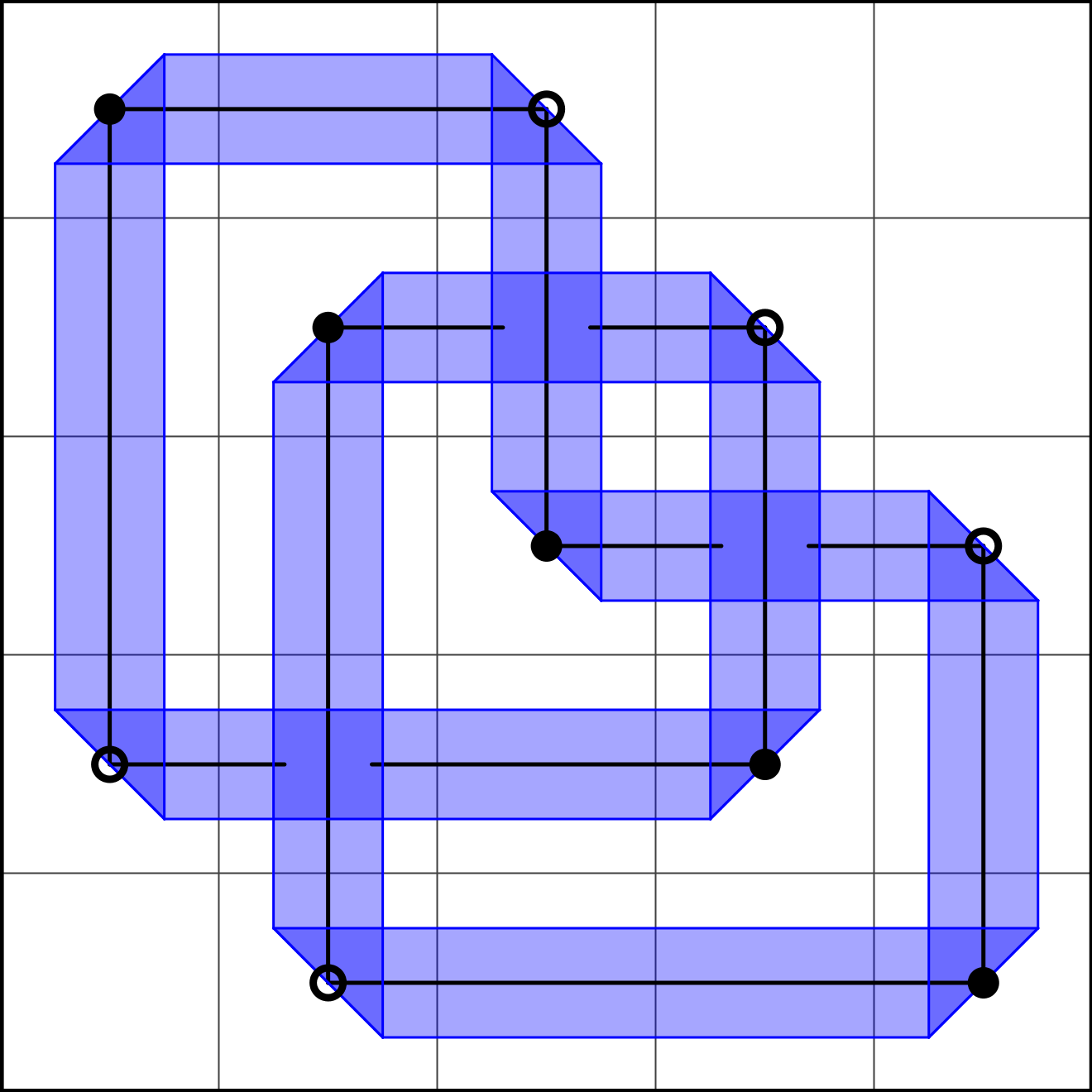}

\medskip

\parbox{5.5in}{\caption{\em Grid diagram, knot diagram, and flat knotted ribbon, respectively.  The width of the ribbon can be as large as $1$, which is the sidelength of the unit square cell.} }

\label{grid}

\end{figure}

We then construct a knot diagram by drawing all horizontal and vertical line segments connecting the white and black dots, with the convention that the vertical strands cross over any horizontal strands that they meet (see Figure 1).

Every knot has a $n \times n$ grid diagram \cite{OSS}, where $n$ is called the grid number. It is straightforward to obtain a flat knotted ribbon from a grid diagram by thickening up all the line segments in the grid diagram and introducing a $45$-degree fold at each white or black dot to make the ribbon turn (see Figure 1). The flat knotted ribbon obtained this way can clearly have a width of one, and the problem is then to calculate the length of the ribbon.  

For each knot, the minimal grid number of its grid diagrams is called the {\em grid index} of the knot. The relation between the grid index and crossing number of a knot has been well studied \cite{BP, HKON}.  It has been proved \cite{BP, HKON} that $g(k) \leq c(K) + 2$ for every nontrivial knot $K$, where $g(K)$ and $c(K)$ denote the grid index and crossing number of $K$, respectively

\begin{theorem}
For every non-trivial knot $K$, the ribbonlength has an upper bound that is quadratic in the crossing number of $K$. More precisely, the upper bound is 
$$2 [c(K) + 1][c(K) + 2],$$ where $c(K)$ is the crossing number of $K$.
\end{theorem} 
\begin{proof}
Suppose the knot $K$ is represented by a $g(K) \times g(K)$ grid diagram, where $g(K)$ is the grid index of $K$. We now construct the flat knotted ribbon with the unit width and estimate the length of the ribbon as follows. Each horizontal distance between a white dot and a black dot cannot exceed $g(K) - 1$. Hence, the sum of all such horizontal distances has an upper bound $g(K)[g(K) - 1]$. Similarly, the sum of all vertical distances between the white and black dots has an upper bound $g(K)[g(K) - 1]$.  Since the ribbonlength is the minimal length of ribbon with unit width, we obtain
$$\mbox{Ribbonlength(K)} \leq  \mbox{Length of ribbon}  \leq 2 g(K)[g(K) - 1]  \leq 2 [c(K) + 1][c(K) + 2] \ ,$$ where we have used $g(k) \leq c(K) + 2$ in the second inequality. The proof of the theorem is completed. \end{proof}

Kusner \cite{K} conjectured that the ribbonlength of a knot has upper and lower bounds which are linear in the crossing number of the knot. We next improve Theorem 1 from a quadratic upper bound to a linear upper bound for all nontrivial torus knots and twist knots.

\subsection{Torus knots}

Torus knots have grid diagrams with clear patterns, so it is possible to use the above construction of a flat knotted ribbon from a grid diagram to calculate the ratio of length to width precisely. The result gives a linear upper bound on the ribbonlength of torus knots. 
\begin{theorem}
For every non-trivial torus knot $K$, the ribbonlength has an upper bound that is linear in the crossing number of $K$. More precisely, the upper bound is 
$8 c(K)$.
\end{theorem} 

\begin{proof}
It is known \cite{A} that a $(p, q)$ torus knot is an unknot if and only if either $p$ or $q$ is $\pm 1$ and that a torus link arises if $p$ and $q$ are not coprimes. 
It is also known that both the $(q, p)$ torus knot and $(-p, -q)$ torus knot are equivalent to the $(p, q)$ torus knot and that the $(p,  - q)$ torus knot is the mirror image of the $(p, q)$ torus knot. Without loss of generality, we may consider the $(p, q)$ torus knot, where $p$ and $q$ are coprime and $2 \leq p < q$. Figure 2 gives a grid diagram for the $(p, q)$ torus knot \cite{LHLO, OSS}.

\begin{figure}[!htbp]
\centering
\includegraphics[width=0.5\textwidth]{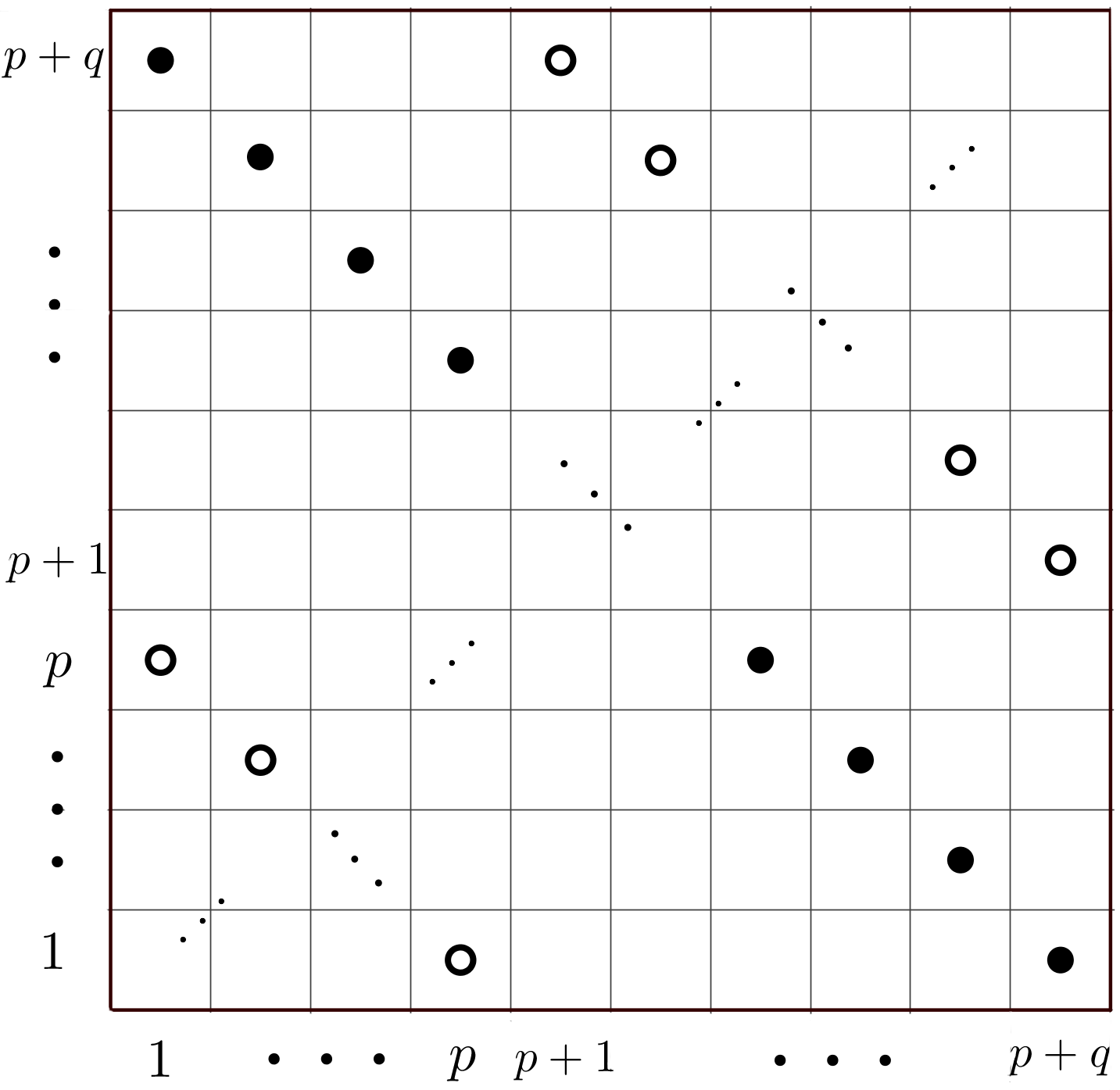}

\medskip

\parbox{5.5in}{\caption{\em Grid diagram of the $(p, q)$ torus knot. The $p^{\text{th}}$ diagonal below and the $q^{\text{th}}$ diagonal above the main diagonal are filled with white dots.}}

\label{pqgrid}

\end{figure}

We now construct the flat knotted ribbon with unit width and calculate the length of the ribbon as follows. The sum of all the horizontal distances between the white and black dots is $2 p q$, and so is the sum of all the vertical distances between the white and black dots.
Since the crossing number of the $(p,q)$ torus knot is $(p-1)q$ in view of $p < q$, we hence obtain
$$ {\mbox{Length of ribbon} \over c(K)} = {4 p q \over (p-1) q} = {4 \over 1 - {1 \over p}} \leq 8 \ ,$$
where we have used $p \geq 2$ in the inequality. Since the ribbonlength is the minimal length of ribbon with unit width,  we  have
$$\mbox{Ribbonlength(K)} \leq \mbox{Length of ribbon} \leq 8 c(K) \  ,$$
which proves the theorem.
\end{proof}

We note that  Kennedy, Mattman, Raya, and Tating \cite{KMRT} found a linear upper bound on the ribbonlength for some families of torus knots and a quadratic upper bound for other families of torus knots. Our results extend their results to a linear upper bound on ribbonlength for all non-trivial torus knots.

\subsection{Twist knots}

In this subsection, we use grid diagrams to study flat knotted ribbons for twist knots, which are considered the simplest type of knots after the torus knots.  A twist knot is obtained by linking together the two ends of a loop with multiple half twists (see Figure 3). We use the notation of the paper \cite{HS} to denote by $J(2, - n)$ and $J(-2, - n)$ the twist knots pictured in Figure 3.  The $n$ crossings are right-handed for $n < 0$ and left-handed for $n > 0$. 

\bigskip

\begin{figure}[!htbp]
\centering
\includegraphics[width=0.4\textwidth]{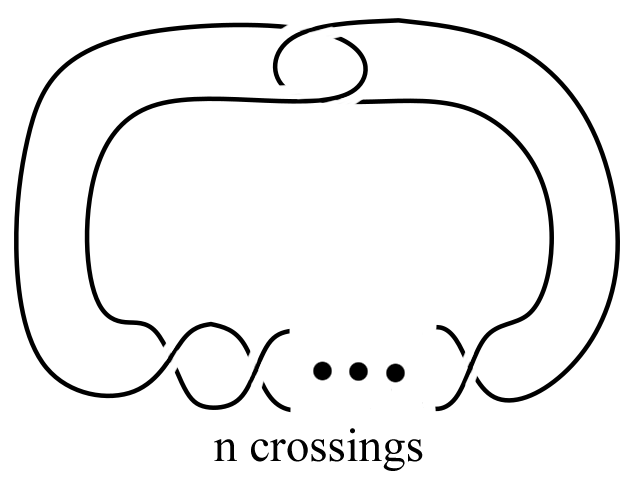}
\includegraphics[width=0.4\textwidth]{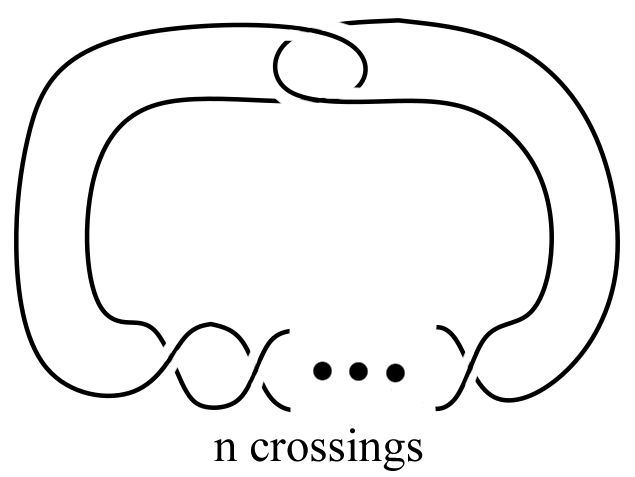}

\medskip

\parbox{5.5in}{\caption{\em The knot diagrams of $J(2, - n)$ and $J(-2, -n)$, $n > 0$.}}

\label{two twist knots}

\end{figure} 

Like torus knots, twist knots also have patternable grid diagrams and their crossing numbers have been known. We can therefore use the method described earlier this section to obtain a linear upper bound on the ribbonlength of twist knots.

\begin{theorem}
For every non-trivial twist knot $K$, the ribbonlength has an upper bound that is linear in the crossing number of $K$. More precisely, the upper bound is $8 c(K)$. 
\end{theorem} 
\begin{proof}

It is known \cite{ DHY, HS} that $J(2, 0)$ is an unknot, that $J(2, n)$ is equivalent to $J(-2, n-1)$, and that $J(-2, -n)$ is the mirror image of $J(2, n)$. It therefore suffices to study the twist knot $J(2, - n)$, where $n$ is a positive integer. Such a twist knot $J(2,- n)$ has the grid diagram as displayed in Figure 4. We can verify that Figure 4 indeed gives a grid diagram of $J(2, -n)$ by drawing the knot diagram, turning the knot diagram $45$ degrees clockwise, and comparing it with Figure 3. 

\bigskip

\begin{figure}[!htbp]
\centering
\includegraphics[width=0.45\textwidth]{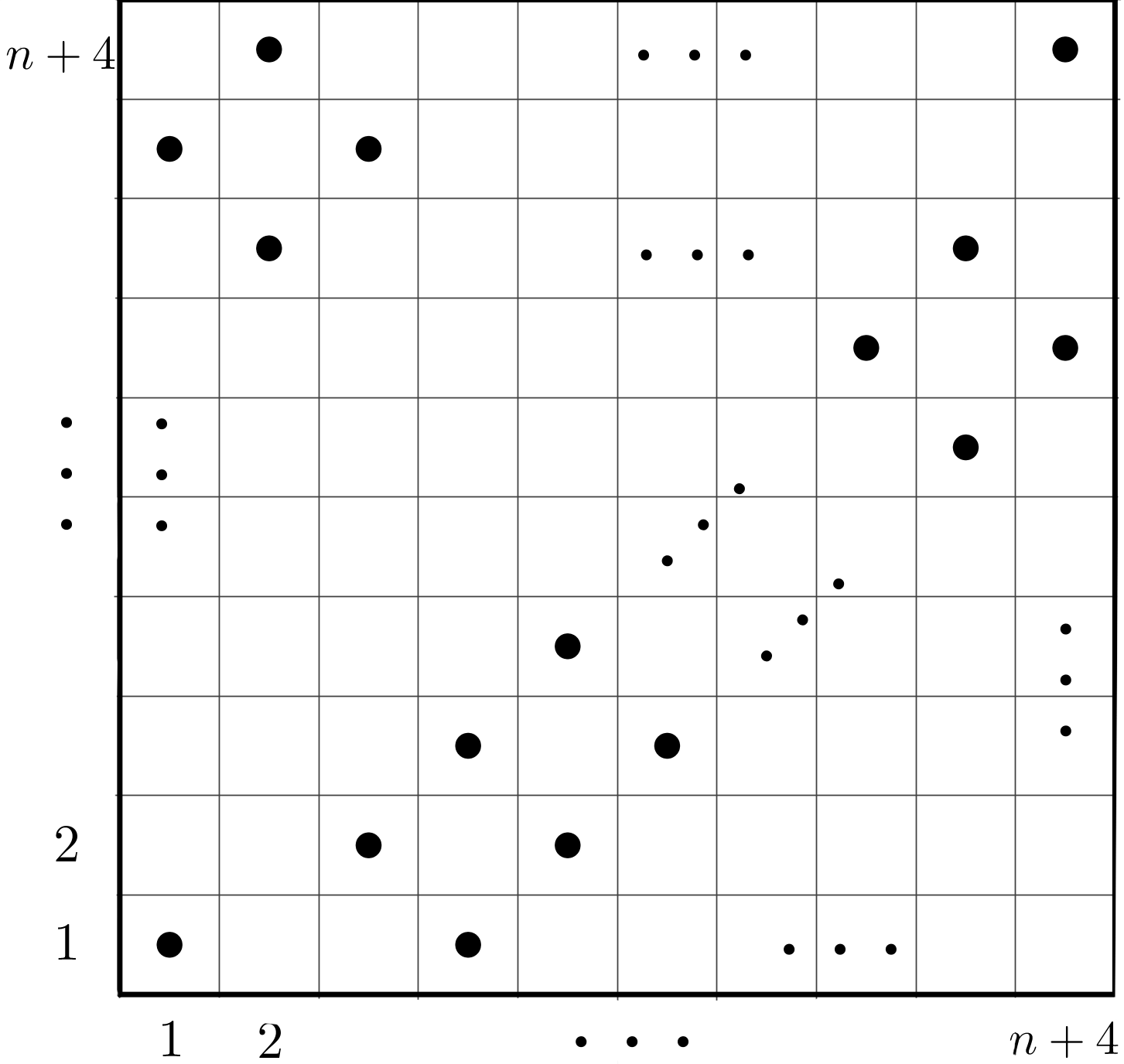}

\medskip

\parbox{5.5in}{\caption{\em Grid diagram of $J(2, - n)$. We omit the distinction between white and black dots because it is not important here.  } }

\label{twistgrid}

\end{figure}

We now construct the flat knotted ribbon with unit width and calculate the length of the ribbon as follows. The sum of all the horizontal distances between dots  is $(n+2) + (n+1) \times 2 + (n+1) + 3 =  (4n +8)$, and so is the sum of all the vertical distances between dots. Since the crossing number of the twist knot $J(2, -n)$ is $n +2 $, we hence obtain
$$ {\mbox{Length of ribbon} \over c(K)} = {2(4n +8) \over n+2} =  8 \ .$$
We therefore have
$$\mbox{Ribbonlength(K)} \leq \mbox{Length of ribbon} = 8 c(K) \  ,$$
which proves the theorem.
\end{proof}

We conclude this section with an observation. Our method of using grid diagrams to construct flat knotted ribbons is new and can clearly be employed to study many other classes of knots. The flat knotted ribbons constructed from grid diagrams have small overlapped portions near the foldings and crossings, and furthermore the ribbons are laid down only horizontally and vertically. To improve the linear upper bound on the ribbonlength, we need to shorten the length of the ribbon by enlarging the overlap in a flat knotted ribbon and allowing the ribbon to go in directions other than horizontal and vertical ones. In the next section, we will use a different method to construct flat knotted ribbons for twist knots. The overlapped portion in the flat knotted ribbon is so large that we get a much better ratio of length to width. Our approach has some similarity with the ones used by previous researchers \cite{K, KMRT}, but it gives excellent results for all twist knots rather than merely some classes of twist knots.

\section{A sharper linear upper bound for twist knots}

In this section, we obtain a better linear upper bound on the ribbonlength by constructing a flat knotted ribbon via folding the ribbon over itself multiple times to minimize the length of the ribbon. This produces a much better ratio of length to width. As discussed in Section 2.2, it suffices to study the twist knot $J(2, - n)$, where $n$ is a positive integer. 

\begin{theorem}
The ribbonlength of the twist knot $J(2, - n)$  has an upper bound of
\begin{equation}
\label{even}
\frac{\sqrt{5}+1}{2} \ n + {9 + \sqrt{5} \over 2} +\sqrt{5+\sqrt{5} \over 2}
\end{equation}
when $n$ is a positive even integer and an upper bound of
\begin{equation}
\label{odd}
\frac{\sqrt{5}+1}{2} \ n +  5 + \sqrt{5}+\sqrt{5+\sqrt{5} \over 2}
\end{equation}
when $n$ is a positive odd integer.

\end{theorem}

The upper bounds on the ribbonlengh in Theorem 4 are connected to the crossing number of the knot as follows. Since the twist knot $J(2, -n)$, where $n$ is a positive integer, has crossing number $n+2$, we combine the upper bounds (\ref{even})-(\ref{odd}) to conclude that 
\begin{eqnarray*}
\mbox{Ribbonlength of $J(2, -n)$} &\leq& \frac{\sqrt{5}+1}{2} \ n +  5 + \sqrt{5}+\sqrt{5+\sqrt{5} \over 2} \\
&=&  \frac{\sqrt{5}+1}{2} \ c(J(2, -n)) +  4 +\sqrt{5+\sqrt{5} \over 2} \\
&\leq& \frac{\sqrt{5}+2}{2} \ c(J(2, -n))
\end{eqnarray*}
for positive integer $n \geq 10$. We note that the linear upper bound  for $n \geq 10$ $$\frac{\sqrt{5}+2}{2} \ c(J(2, -n))  \approx 2.12 \ c(J(2, -n))$$ is a big improvement of the linear upper bound $$8 c(J(2, -n))$$ that is obtained in Theorem 3.

We first prove Theorem 4 for the case when the number of half twists is an even integer, i.e., $n$ is even. We then extend the result to the case when the number of half twists is an odd integer.

\begin{figure}[!htbp]
\centering
\includegraphics[width=0.8\textwidth]{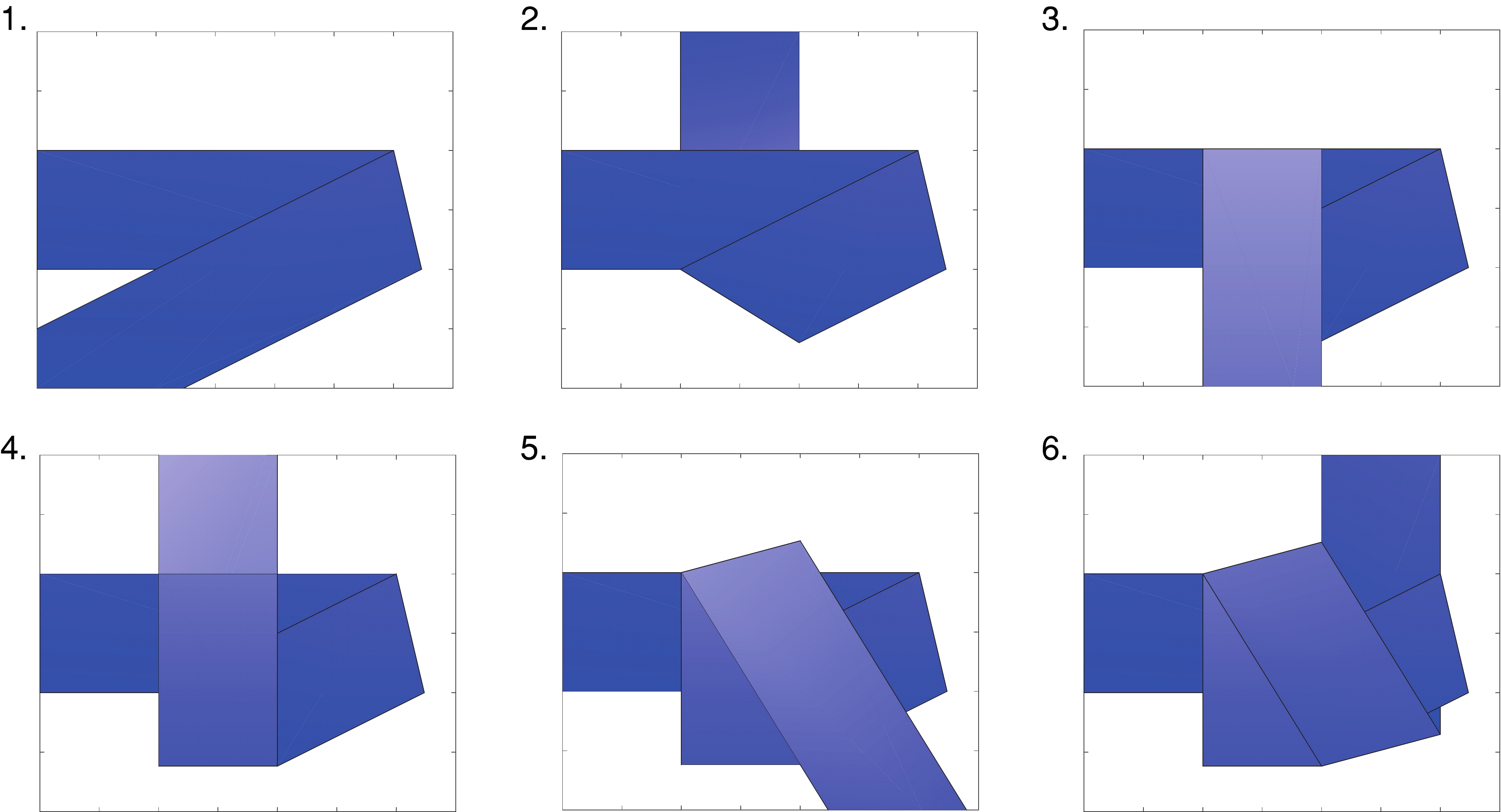}

\bigskip

\parbox{5.5in}{\caption{\em Steps to construct a flat  knotted ribbon for a twist knot. For simplicity, the final step of closing the ribbon is not shown in the figure.}}
\label{steps}

\end{figure} 

We construct a flat knotted ribbon as follows (see Figures 5 and 8). The ribbon first travels horizontally along the center-line $x_1x_2$ and it turns around at the overfold $DG$. The ribbon then goes along the center-line $x_2x_3$ and then turns around at the underfold $OI$. It goes up vertically along $x_3x_4$ to the overfold $AB$ and then turns around to go down vertically along $x_4x_5$ to the underfold $HI$. The ribbon then goes up along $x_5x_4$ to the overfold $AB$ and then goes down along $x_4x_5$ to the underfold $HI$ multiple times. The ribbon then goes up vertically along $x_5x_6$ to the overfold $AC$, turns around to travel along the center-line $x_6x_7$ on top of the previous ribbon to the underfold $IJ$, and goes up vertically along $x_7x_8$ under the rest of ribbon except the one along $x_1x_2$ to the edge $BD$. We finally close the ribbon by folding the ribbon down at along $BD$, then immediately folding the ribbon again along $DE$, and letting the ribbon travel horizontally to its start line $AO$.

\begin{figure}[!htbp]
\centering
\includegraphics[width=0.7\textwidth]{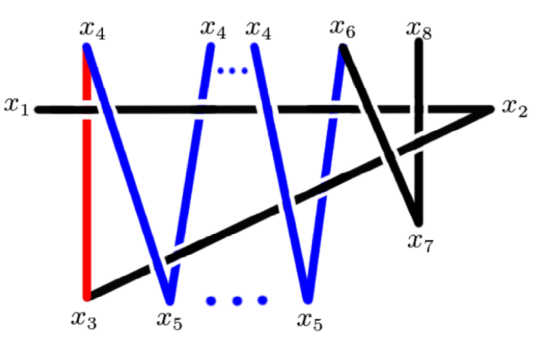}
\parbox{5.5in}{\caption{\em A slight distortion of the center-line of our constructed flat knotted ribbon, in which the line segment $x_4x_5$ should overlap the line segment $x_3x_4$. For simplicity, the linking of $x_1$ and $x_8$ is not shown in the figure.  }}
\label{distortion}

\end{figure}

\bigskip

We list several examples of the paths of the center-line of the flat knotted ribbons. The square brackets indicate the place where the ribbon has multiple half twists.
\begin{itemize}
\item 2 half twists: $x_1x_2[x_3x_6x_7]x_8x_9x_1$
\item 4 half twists: $x_1x_2[x_3x_4x_5x_6x_7]x_8x_9x_1$
\item 6 half twists: $x_1x_2[x_3x_4x_5x_4x_5x_6x_7]x_8x_9x_1$
\item $n$ half twists ($n \ge 4$): $x_1x_2[x_3\underbrace{x_4x_5}_{\text{\normalfont  ${ n-2 \over 2}$ times}}x_6x_7]x_8x_9x_1$

\end{itemize}

We now show that the flat knotted ribbon that we have constructed corresponds to the twist knot $J(2, - n)$, where $n$ is the number of half twists. Without loss of generality, we consider the case of $4$ half twists: $x_1x_2[x_3x_4x_5x_6x_7]x_8x_9x_1$. The center-line of the flat knotted ribbon is depicted in Figure 6 and it becomes a knot diagram if we link the initial point $x_1$ and the terminal point $x_8$. This knot diagram is obviously equivalent to the leftmost knot diagram in Figure 7. Reidemeister moves displayed in Figure 7 show that the leftmost knot diagram is equivalent to the rightmost knot diagram, which is a knot diagram of the twist knot $J(2, -4)$ (see Figure 3).

\bigskip

\begin{figure}[!htbp]
\centering
\includegraphics[width=1\textwidth]{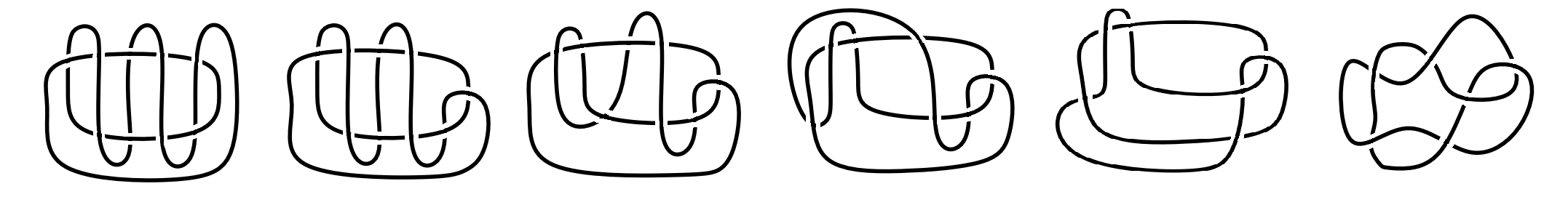}
\parbox{5.5in}{\caption{\em Reidemeister moves from our leftmost knot digram with four half twists to the rightmost knot diagram of the twist knot $J(2, -4)$.}}
\label{Rei}

\end{figure} 

We are now ready to finish the proof of the theorem by calculating the ratio of length to width of the flat knotted ribbon for an even number of half twists and extending the result to the case when the number of half twists is an odd number.

\begin{figure}[h]

\centering
\includegraphics[width=1\textwidth]{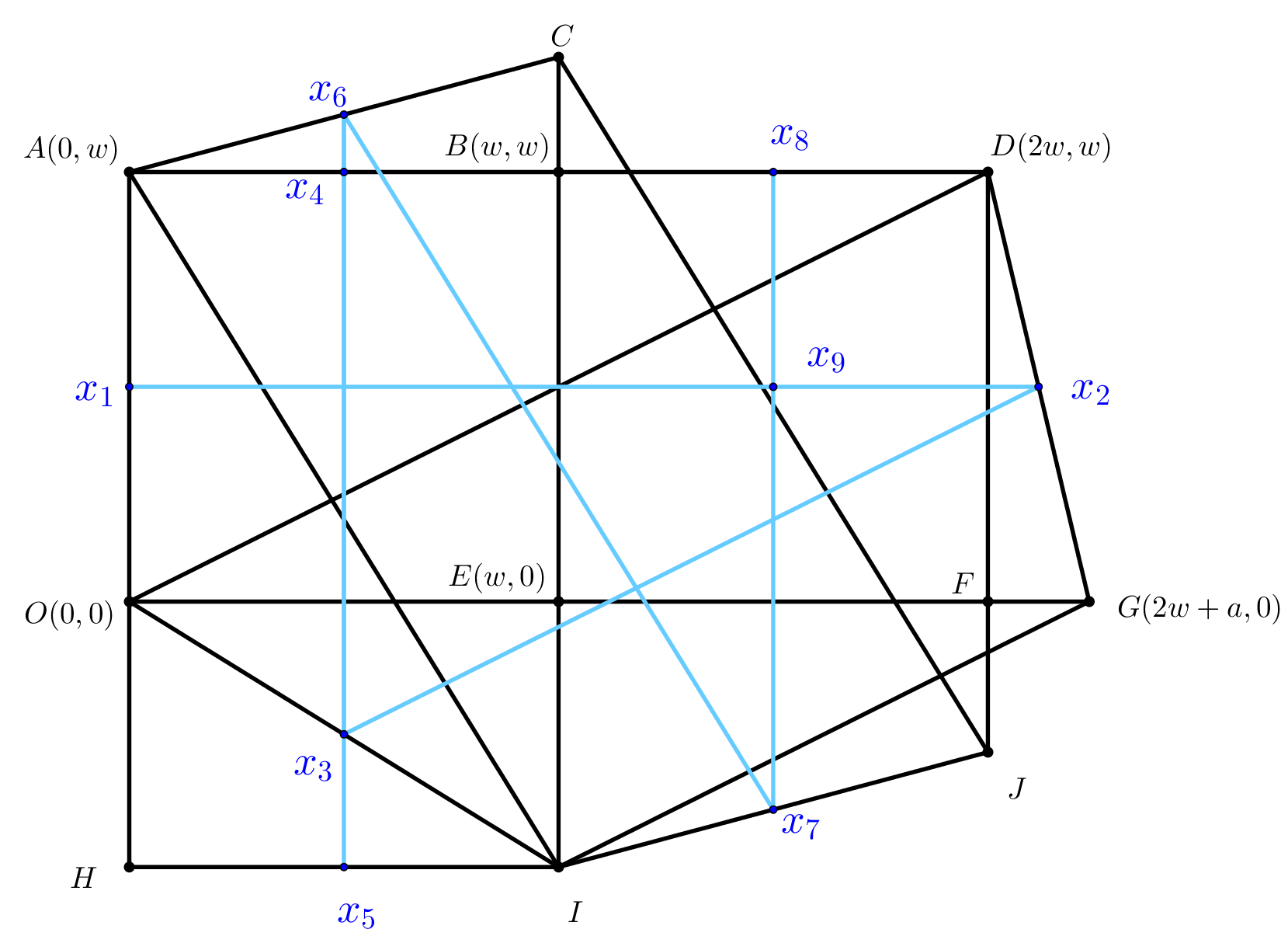}

\parbox{5.5in}{\caption{\em Flat knotted ribbon for $J(2, - n)$ when $n$ is a positive even integer. The line segments $x_1x_2, \ x_2x_3, \ \ldots , \ x_8x_9, \ x_9x_1$ correspond to the centerline of the ribbon.}}

\label{mainfigure}

\end{figure}

\noindent
{\em Proof of Theorem 4:}

To evaluate the ratio of length to width, we denote the ribbon width by $w$ and calculate the length of the ribbon along the center-line from the initial point $x_1$ to $x_8$, then to $x_9$, and finally to the initial point $x_1$. We first set up a coordinate system such that $O = (0,0), \ A = (0,w), \ B = (w,w), \ D = (2w,w),\  E = (w,0), \  F = (2w,0)$ (see Figure 8). We next calculate the coordinates of the rest of points in Figure 8.

Let $\theta_1$ be $\angle x_1x_2x_3$. Since the center-line $x_1x_2$ and ribbon edge $AD$ are parallel and the center-line $x_2x_3$ and ribbon edge $DO$ are also parallel, we have $\angle x_1x_2x_3 = \angle ADO$. Hence, $\tan \theta_1 = \tan (\angle ADO) = {1 \over 2}$.

Let $a = FG$ and thus point $G$ has coordinates $(2w+a, 0)$. We can calculate $a$ exactly as follows. Since $\angle OGD = \angle x_1x_2D = {\pi - \theta_1 \over 2}$, we have $\cot({\pi - \theta_1 \over 2}) = {a \over w}$ and hence $a = w \tan ({\theta_1 \over 2})$. In view of $\tan \theta_1 = {1 \over 2}$, we use the identity $\tan \theta_1 = 2 \tan ({\theta_1 \over 2}) / [1 - \tan^2 ({\theta_1 \over 2})]$ to obtain $\tan ({\theta_1 \over 2}) = \sqrt{5} - 2$ . Therefore, $a = (\sqrt{5} -2) w$.

Because $\angle EGI = \theta_1$, we have $\tan (\angle EGI) = \tan \theta_1 = {1 \over 2}$. Hence,  $IE = {EG \over 2} = {w+a \over 2}$. Points $H$ and $I$ have coordinates $(0, - {w +a \over 2})$ and $(w, - {w+a \over 2})$, respectively.

Let $\theta_2$ be $\angle x_5x_6x_7$. Then $\angle BAC = {\theta_2 \over 2}$, $BC = w\tan({\theta_2 \over 2})$, and point C has coordinates $(w, w+w\tan({\theta_2 \over 2}))$. The angle $\theta_2$ can also be calculated exactly, but its exact value is not needed for future calculations.

Let point $K$ be on line segment $EI$ such that $IK \perp KJ$. Since $\angle x_6x_7x_8 = \theta_2$, we have $\angle KJI = {\theta_2 \over 2}$ and thus $KI = w \tan ({\theta_2 \over 2})$. Hence, point J has coordinates ($2w$, $-\frac{w+a}{2} + w\tan({\theta_2 \over 2}))$.

The points $x_1, \ x_2, \ \ldots , \ x_8$ are the midpoints of the fold line segments and their coordinates can be calculated easily: $x_1 = (0, \frac{w}{2}), \  x_2 = (2w + \frac{a}{2}, \frac{w}{2})), \ x_3 = (\frac{w}{2}, -\frac{w+a}{4})$, \ $x_4 = (\frac{w}{2}, w), \ x_5 = (\frac{w}{2}, - \frac{w+a}{2}), \ x_6 = ({w \over 2}, w+\frac{1}{2}w\tan(\frac{\theta_2}{2})), \ x_7 = (\frac{3}{2}w, \frac{1}{2}w\tan(\frac{\theta_2}{2})-\frac{w + a}{2}),$ and $x_8 = (\frac{3}{2}w, w)$. The last fold line segment $DE$ is not shown in Figure 8 and its midpoint is $x_9 = ({3 w \over2}, {w \over 2})$.

For $n \ge4$, the center line of the ribbon starts at $x_1$,  goes to $x_2$, $x_3$, then travels along $x_4$ to $x_5$ or $x_5$ to $x_4$ exactly $(n-3)$ times, and finally proceeds to $x_6, \ x_7, \ x_8, \ x_9$ before closing up at the initial point $x_1$.
Thus the length of the ribbon is obtained by summing the lengths $x_1x_2, \ x_2x_3, \ x_3x_4$, \ $x_5x_6, \ x_6x_7, \ x_7x_8, \ x_8x_9, \ x_9x_1$ and $(n-3)$ multiples of $x_4x_5$.
\begin{eqnarray*}
\mbox{Length} &=& x_1x_2 + x_2x_3 + x_3x_4 + \left(n-3\right) x_4x_5 + x_5x_6 + x_6x_7 + x_7x_8 + x_8x_9 + x_9x_1 \\
&=& 2w + {a \over 2} + \sqrt{\left({3w + a \over 2}\right)^2 + \left({3w + a \over 4}\right)^2} + {5w + a \over 4} + \left(n-3\right) {3w + a \over 2} \\
&& + \ {w \tan\left({\theta_2 \over 2}\right) + 3w + a \over 2} + \sqrt{w^2 + \left({3w + a \over 2}\right)^2} \\
&& + \ {3w + a - w \tan \left( {\theta_2 \over 2} \right) \over 2}  + {w \over 2} + {3w \over 2} \\
&=& \left[6 + 2 \sqrt{5} + \sqrt{ { 5 + \sqrt{5} \over 2}} + \left(n-3\right) {1 + \sqrt{5} \over 2} \ \right] w \\
&=& \left[ \frac{\sqrt{5}+1}{2} \ n + {9 + \sqrt{5} \over 2} +\sqrt{5+\sqrt{5} \over 2} \ \ \right] w \ ,
\end{eqnarray*}
where we have used $a = (\sqrt{5} - 2) w$. This proves (\ref{even}) for even integer $n > 2$.

We now show that (\ref{even}) is also true for $n = 2$. Since $x_3x_6 = x_3x_4 + x_4x_5 + x_5x_6 - 2 x_4x_5$ (see Figure 8), the length of the ribbon when $n=2$ is equal to
\begin{eqnarray*}
\lefteqn{x_1x_2 + x_2x_3 + x_3x_6 + x_6x_7 + x_7x_8 + x_8x_9 + x_9x_1} \\
&=& x_1x_2 + x_2x_3 + x_3x_4 + (4-3) x_4x_5 + x_5x_6 + x_6x_7 + x_7x_8 + x_8x_9 + x_9x_1 
 - 2 x_4x_5 \\
&=& \left[ \frac{\sqrt{5}+1}{2} \times 4 + {9 + \sqrt{5} \over 2} +\sqrt{5+\sqrt{5} \over 2} \ \ \right] w - 2 \ \left[{\sqrt{5} + 1 \over 2}\right] w \\
&=& \left[ \frac{\sqrt{5}+1}{2} \times 2 + {9 + \sqrt{5} \over 2} +\sqrt{5+\sqrt{5} \over 2} \ \right] w \ ,
\end{eqnarray*}
where we have used (1) for $n=4$ in the second equality. This completes the proof of (\ref{even}).
\begin{figure}[!htbp]
\centering
\includegraphics[width=0.6\textwidth]{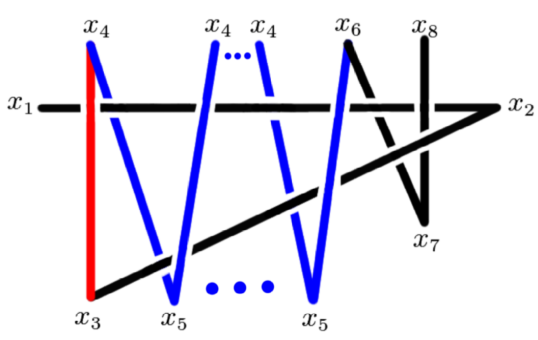}
\parbox{5.5in}{\caption{\em A slight distortion of the center-line of our modified flat knotted ribbon, in which the line segment $x_4x_5$ should overlap the line segment $x_3x_4$. For simplicity,  the linking of $x_1$ and $x_8$ is not shown in the figure. }}
\label{distortion2}

\end{figure}

We finally show how to prove (\ref{odd}) by modifying the flat knotted ribbon. For a positive even integer $n$, we construct the new flat knotted ribbon in a similar way, but we change all the  overpassings and underpassings except the last two passings when $x_7x_8$ still underpasses $x_2x_3$ and overpasses $x_1x_2$ (see Figure 6 for the old knot and Figure 9 for the new knot). 

The Reidemeister moves similar to those displayed in Figure 7 show that the new knot is $J(2, n)$, where the $n$ crossings are now right-handed. The calculations of the ratio of length to width for the old knot $J(2, -n)$ also work for the new knot $J(2, n)$. Hence, the ribbonlength $J(2, n)$ also has an upper bound (\ref{even}) when $n$ is a positive even integer. Since $J(2, m)$ is equivalent to $J(-2, m-1)$ and $J(-2, -m)$ is the mirror image of $J(2, m)$ \cite{HS}, we see that $J(2, n)$ is the mirror image of $J(2, -n +1)$. Consequently, the ribbonlength of the twist knot $J(2, - (n-1))$ has an upper bound 
$$\frac{\sqrt{5}+1}{2} \ n + {9 + \sqrt{5} \over 2} +\sqrt{5+\sqrt{5} \over 2} = \frac{\sqrt{5}+1}{2} \ (n - 1) +  5 + \sqrt{5} + \sqrt{5+\sqrt{5} \over 2} \ .$$
This proves (\ref{odd}) and the proof of Theorem 4 is completed. 
\hfill $\square$

\bigskip

\section{Conclusion and future directions}

We use grid diagrams to construct the flat knotted ribbons and prove a quadratic upper bound on the ribbonlength for every non-trivial knot. We improve the quadratic upper bound to a linear upper bound for every non-trivial torus knot and twist knot. We obtain a sharper linear upper bound for twist knots by constructing a flat knotted ribbon via folding the ribbon over itself multiple times.

Our approach of using grid diagrams to study flat knotted ribbons is novel and can likely be used to obtain a linear upper bound on the ribbonlength for more general knots. One possible direction for future research would be to use grid diagrams to prove a linear upper bound on the ribbonlength for more complicated families of knots such as rational knots. 

Another possible direction would be to improve Theorem 1 by studying a grid diagram of a knot $K$ in a $g(K) \times g(K)$ grid, where $g(K)$ is the grid index of $K$. If we can show that the length of the ribbon in such a grid diagram has an upper bound that is linear in $g(K)$,  then the ribbonlength, in view of $g(K) \leq c(K) + 2$, would have an upper bound that is linear in $c(K)$.

\section{Acknowledgments} 

I would like to thank my mentors Mr. Vishal Patil of MIT, Professor Sergei Chmutov of the Ohio State University, and Mr.  Daniel Vitek of Princeton University for their invaluable help and support.

I am grateful to Professor David Jerison of MIT, Dr. Tanya Khovanova of MIT,  Professor Ankur Moitra of MIT, and Professor John Rickert of Rose-Hulman Institute of Technology for their helpful comments and suggestions.

This research is supported in part by the Research Science Institute, PRIMES-USA, and the Department of Mathematics, MIT.
I would also like to acknowledge my sponsors Mr. William Madar, the chairman of the Nordson Corporation, Dr. Srinivasan Dasarthy, Professor Javidhya Dasarthy, Dr. Sivaswami Sivaraman, and Dr. Vidya Raman for their generous financial support.

\end{document}